\documentclass[11pt]{amsart}

\usepackage{amssymb}
\usepackage{amsthm}
\usepackage{amsmath}
\usepackage{graphicx}
\usepackage{amsaddr}
\usepackage{comment}
\usepackage[usenames,dvipsnames]{xcolor}
\usepackage[margin = 1in]{geometry}
\usepackage{enumerate}

\usepackage{lineno}

\usepackage{lipsum}
\usepackage{amsfonts}
\usepackage{graphicx}
\usepackage{epstopdf}
\usepackage{algorithmic}
\ifpdf
  \DeclareGraphicsExtensions{.eps,.pdf,.png,.jpg}
\else
  \DeclareGraphicsExtensions{.eps}
\fi

\usepackage{amsopn}

\usepackage{enumerate}
 \newtheorem{thm}{Theorem}[section]
 
 \newtheorem{lem}[thm]{Lemma}
 \newtheorem{prop}[thm]{Proposition}
\newtheorem{clm}[thm]{Claim}

 \theoremstyle{definition}
 \newtheorem{defn}{Definition}

 \theoremstyle{remark}
 \newtheorem{rem}{Remark}

 \numberwithin{equation}{section}

\newcommand{\per}{\text{per}}

\newcommand{\ud}{\,\mathrm{d}}

\newcommand{\Hper}{H_{\per}}
\newcommand{\Lper}{L_{\per}}
\newcommand{\Lty}{L^{\infty}}
\newcommand{\uv}{u_{\varepsilon}}
\newcommand{\uvt}{u_{\varepsilon t}}
\newcommand{\uvk}{u_{\varepsilon_k}}

\newcommand{\bdt}{b_\delta(t)}

\newcommand{\Dld}{D_{1\delta}(t)}
\newcommand{\Dgd}{D_{2\delta}(t)}
\newcommand{\vark}{\varepsilon_k}

\newcommand{\dpstyle}{\displaystyle}

\newcommand{\ale}{a.e.\,\,}

\newcommand{\nn}{\nonumber}

\newcommand{\blue}{ }

\begin{document}

\title[Weak Solution of a continuum model for vicinal surface in the ADL regime]{Weak solution of a
  continuum model for vicinal surface in the attachment-detachment-limited regime}

\author{Yuan Gao}
\address{School of
Mathematical Sciences\\
   Fudan  University,
Shanghai 200433, P.R.\ China\\Department of Mathematics and Department of
  Physics\\Duke University,
  Durham NC 27708, USA}
\email{gaoyuan12@fudan.edu.cn}
\author{Jian-Guo Liu}
\address{Department of Mathematics and Department of
  Physics\\Duke University,
  Durham NC 27708, USA}
\email{jliu@phy.duke.edu}

\author{Jianfeng Lu}
\address{Department of Mathematics, Department of
  Physics, and Department of Chemistry\\Duke University, Box 90320,
  Durham NC 27708, USA}
\email{jianfeng@math.duke.edu}

\date{\today}

\begin{abstract}
  We study in this work a continuum model derived from 1D
  attachment-detachment-limited (ADL) type step flow on vicinal
  surface,
  \begin{equation*}
    u_t=-u^2(u^3)_{hhhh},
  \end{equation*}
  where $u$, considered as a function of step height $h$, is the step
  slope of the surface.  We formulate a notion of weak
  solution to this continuum model and prove the existence of a global
  weak solution, which is positive almost everywhere. We also
  study the long time behavior of weak solution and prove it converges
  to a constant solution as time goes to infinity.  The space-time
  H\"older continuity of the weak solution is also discussed as a
  byproduct.
\end{abstract}



\maketitle

\section{Introduction}
During the heteroepitaxial growth of thin films, the evolution of the
crystal surfaces involves various structures.  Below the roughening
transition temperature, the crystal surface can be well characterized
as steps and terraces, together with adatoms on the terraces. Adatoms
detach from steps, diffuse on the terraces until they meet one of the
steps and reattach again, which lead to a step flow on the crystal
surface. The evolution of individual steps is described mathematically
by the Burton-Cabrera-Frank (BCF) type models \cite{BCF}; see
\cite{Duport1995a, Duport1995b} for extensions to include elastic
effects.  Denote the step locations at time $t$ by
$x_i(t), i \in \mathbb{Z}$, where $i$ is the index of the
steps. Denote the height of each step as $a$. For one dimensional
vicinal surface (i.e., monotone surface), if we do not consider the
deposition flux, the original BCF type model, after
non-dimensionalization, can be written as (we set some physical
constants to be $1$ for simplicity):
\begin{equation}\label{BCF}
  \dot{x}_i=\frac{D}{ka^2}\bigg(\frac{\mu_{i+1}-\mu_i}{x_{i+1}-x_i+\frac{D}{k}}
  -\frac{\mu_{i}-\mu_{i-1}}{x_i-x_{i-1}+\frac{D}{k}}\bigg), \text{ for }1\leq i\leq N.
\end{equation}
where $D$ is the terrace diffusion constant, $k$ is the hopping rate of an adatom to the upward or downward step, and $\mu$ is the chemical potential whose expression ranges under different assumption.
 Often two limiting cases of the classical BCF type model \eqref{BCF} were considered.
See \cite{Tersoff1995,Tersoff1998} for diffusion-limited (DL) case and see \cite{Isr2000,She2011} for attachment-detachment-limited (ADL) case.

In DL regime, the dominated dynamics is diffusion across the terraces, i.e. $\frac{D}{k}<< x_{i+1}-x_i$, so the step-flow ODE becomes
\begin{equation}\label{ODE1}
  \dot{x}_i=\frac{D}{ka^2}\bigg(\frac{\mu_{i+1}-\mu_i}{x_{i+1}-x_i}-\frac{\mu_{i}-\mu_{i-1}}{x_i-x_{i-1}}\bigg), \text{ for }1\leq i\leq N.
\end{equation}

In ADL regime, the diffusion across the terraces is fast, i.e. $\frac{D}{k}>>x_{i+1}-x_i$, so the dominated processes are the exchange of atoms at steps edges, i.e., attachment and detachment. The step-flow ODE in ADL regime becomes
\begin{equation}\label{ODE}
  \dot{x_i}=\frac{1}{a^2}\big(\mu_{i+1}-2\mu_i+\mu_{i-1}\big), \text{ for }1\leq i\leq N.
\end{equation}

 Those models are widely used for crystal growth of thin films
on substrates; see many scientific and engineering applications in the books
\cite{PimpinelliVillain:98, WeeksGilmer:79, Zangwill:88}.
As many of the film's properties and performances originate in their growth processes, understanding and mastering thin film growth is one of the major challenges of materials science.

Although these mesoscopic models provide details of discrete nature,
continuum approximation for the discrete models is also used to
analyze the step motion. {\blue They involve fewer variables than
discrete models so they can reveal the leading physics structure and are easier for numerical simulation.}  Many interesting continuum models can be
found in the literature on surface morphological evolution; see
\cite{Zang1990,Tang1997,Yip2001,Xiang2002,Xiang2004,Shenoy2002,Margetis2011,Leoni2014,our}
for one dimensional models and \cite{Margetis2006, Xiang2009} for two
dimensional models. The study of relation between the discrete ODE
models and the corresponding continuum PDE has raised lots of
interest. Driven by this goal, it is important to understand the
well-posedness and properties of the solutions to those continuum
models.

For a general surface with peaks and valleys, the analysis of step
motion on the level of continuous PDE is complicated so we focus on a
simpler situation in this work: a monotone one-dimensional step train,
known as the vicial surface in physics literature.  In this case,
\textsc{Ozdemir, Zangwill} \cite{Zang1990} and \textsc{Al Hajj
  Shehadeh, Kohn and Weare} \cite{She2011} realized using the step
slope as a new variable is a convenient way to derive the continuum
PDE model
\begin{equation}\label{PDE}
  u_t=-u^2(u^3)_{hhhh},
\end{equation}
where $u$, considered as a function of step height $h$, is the step
slope of the surface.  We validate this continuum model by formulating
a notion of weak solution. Then we prove the existence of such a weak
solution. The weak solution is also persistent, i.e., it is positive
(or negative) almost everywhere if non-negative (or non-positive)
initial data are assumed.

The starting point of this PDE is the 1D
attachment-detachment-limited (ADL) type models \eqref{ODE}.
To simplify the analysis,
 we will consider a periodic
train of steps in this work, i.e., we assume that
\begin{equation}\label{1}
  x_{i+N}(t) - x_{i}(t) = L, \qquad \forall\,i\in \mathbb{Z},\,\forall\,t\geq 0,
\end{equation}
where $L$ is a fixed length of the period. Thus, only the step
locations in one period $\{x_i(t), \, i = 1, \ldots, N\}$ are
considered as degrees of freedom. Since the vicinal surface is very
large in practice from the microscopic point of view, this is a good
approximation.  We set the height of each step as $a = \frac{1}{N}$,
and thus the total height changes across the $N$ steps in one period
is given by $1$. This choice is suitable for the continuum limit
$N\rightarrow \infty$.  See Figure 1 for an example of step train in
one period.
\begin{figure}[htbp]
\includegraphics[height=1.8in,width=5in]{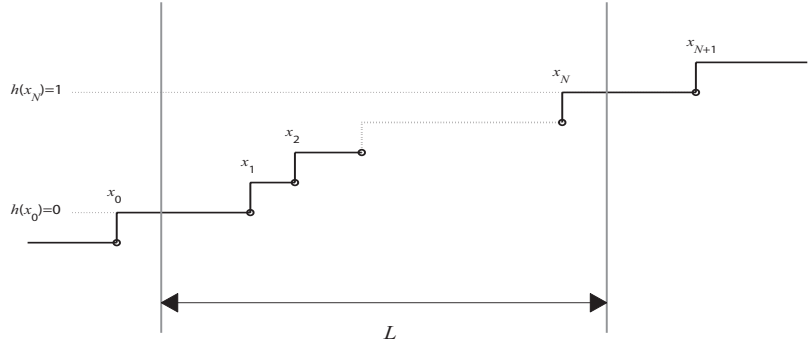}
\caption{ An example of step configurations with periodic boundary condition.}
\end{figure}

The general form of the (free) energy functional due to step
interaction is\footnote{In this work, we neglect long range elastic
  interactions between the steps in the model; related models with
  long range elastic interactions are briefly discussed below in later
  part of the introduction.}
\begin{equation}\label{Ff}
  F_N=a\sum_{i=0}^{N-1} f\big(\frac{x_{i+1}-x_i}{a}\big),
\end{equation}
where $f$ reflects the physics of step interaction.  Following the
convention in focusing on entropic and elastic-dipoles interaction
\cite{SSR1,SSR2}, we choose $f(r)=\frac{1}{2r^2}$. Hence each step
evolves by \eqref{ODE} with chemical potential $\mu_i$ defined as the
first variation of the step interaction energy
\begin{equation}\label{F_N}
F_N=\frac{1}{2}\sum_{i=0}^{N-1}\frac{a^3}{(x_{i+1}-x_i)^2},
\end{equation}
with respect to $x_i.$
That is
\begin{equation}\label{cp1}
  \mu_i=\frac{1}{a}\frac{\partial F_N}{\partial x_i}=\frac{a^2}{(x_{i+1}-x_i)^3}-\frac{a^2}{(x_i-x_{i-1})^3}, \text{ for }1\leq i\leq N.
\end{equation}
From the periodicity of $x_i$ in \eqref{1}, it is easy to see the periodicity of $\mu$ such that $\mu_i=\mu_{i+N}.$

When the step height $a \to 0$ or equivalently, the number of steps in
one period $N \to \infty$, from the viewpoint of surface slope,
\textsc{Al Hajj Shehadeh, Kohn and Weare} \cite{She2011} and
\textsc{Margetis, Nakamura} \cite{Margetis2011} studied the continuum
model \eqref{PDE}; see also \cite{Zang1990} for physical derivation in
general case. We recall their ideas in our periodic setup. Denote the
step slopes as
\begin{equation}\nn
u_i(t)=\frac{a}{x_{i+1}(t)-x_i(t)}, \text{ for }1\leq i\leq N.
\end{equation}
The periodicity of $x_i$ in \eqref{1} directly implies the periodicity of $u_i$, i.e. $u_i=u_{i+N}.$
Then by straight-forward calculation, we have the ODE for slopes
\begin{equation}\label{yODE}
  \dot{u}_i=-\frac{1}{a^4}u_i^2\Big[(u_{i+2}^3-2u_{i+1}^3+u_i^3)-2(u_{i+1}^3-2u_{i}^3+u_{i-1}^3)+(u_{i}^3-2u_{i-1}^3+u_{i-2}^3)\Big].
\end{equation}
Under the periodic setup, when considering step slope $u$ as a
function of $h$ in continuum model, $u$ has period $1$.  Keep in mind
the height of each step $x_i$ is $h_i=ia.$ It is natural to anticipate
that as $N\rightarrow \infty,$ the solution of the slope ODE
\eqref{yODE} should converge to the solution $u(h,t)$ of continuum
model \eqref{PDE}, which is 1-periodic with respect to step height
$h$.


By different methods, \cite{She2011} and \cite{Margetis2011}
separately studied the self-similar solution of ODE \eqref{yODE} and
PDE \eqref{PDE}. For monotone initial data, i.e.
$x_1(0)<x_2(0)<\cdots <x_N(0)$, \cite{She2011} proved the steps do not
collide and the global-in-time solution to ODE \eqref{yODE} (as well
as ODE \eqref{ODE}) was obtained in their paper. By introducing a
similarity variable, \cite{She2011} first discovered that the
self-similar solution is a critical point of a ``similarity energy'',
for both discrete and continuum systems. Then they rigorously prove
the continuum limit of self-similar solution and obtained the
convergence rate for self-similar solution.

However, as far as we know, the global-in-time validation of the
time-dependent continuum limit model \eqref{PDE} is still an open
question as stated in \cite{Kohnbook}. In fact, it is not even known
whether \eqref{PDE} has a well-defined, unique solution. Although the
positivity of solution to continuum model \eqref{PDE} corresponds to
the non-collision of steps in discrete model, which was proved in
\cite{She2011}; even a ``formal proof'' of positive global weak
solution in the time-dependent continuous setting has not been
established.

Our goal is to formulate a notion of weak solution and prove the existence of global weak solution. We also prove the almost everywhere positivity of the solution, which might help the study of global convergence of discrete model \eqref{ODE} to its continuum limit \eqref{PDE} in the future.
Moreover, we study the long time behavior of weak solutions and prove that all weak solutions converge to a constant as time goes to infinity.
The space-time H\"older continuity of the solution is also obtained.

{ \blue One
  of the key structures of the model is that it possesses the following two Lyapunov
  functions,
  \begin{equation}\label{F_u}
  F{\blue (u)}:=\frac{1}{2}\int_0^1 u^2 \ud h,
\end{equation}
and
\begin{equation}\label{E_u}
  E{\blue (u)}:=\int_0^1 \frac{1}{6}[(u^3)_{hh}]^2 \ud h.
\end{equation}
Then we have
\begin{equation}
  \frac{\delta F{\blue (u)} }{\delta u}=u,\quad
  \frac{\delta E{\blue (u)}}{\delta u}=u^2(u^3)_{hhhh},\nn
\end{equation}
and \eqref{PDE} can be recast as
\begin{equation}\label{EE1222}
  u_t=-\frac{\delta E{\blue (u)}}{\delta u}=-u^2\partial_{hhhh}\Big(u^2 \frac{\delta F{\blue (u)}}{\delta u}\Big).
\end{equation}

Since the homogeneous degree of $E(u)$ is $6$, one has
$$6E(u)=\int_0^1 u\frac{\delta E(u)}{\delta u}\ud h.$$
Then by \eqref{EE1222},we obtain
\begin{equation}\label{1221_2}
  \frac{\ud F(u)}{\ud t}+6 E(u)=\int_0^1 u\big(u_t+\frac{\delta E(u)}{\delta u}\big)\ud h=0.
\end{equation}
Notice that
\begin{equation}\label{1220E}
\frac{\ud E{\blue (u)}}{\ud t}=\int_0^1 \frac{\delta E{\blue (u)}}{\delta u} u_t \ud h=-\int_0^1 u_t^2\ud h \leq 0.
\end{equation}
Therefore, we also have the following dissipation structures:
\begin{equation}\label{1221_1}
  \frac{\ud E(u)}{\ud t}+D(u)=0,
\end{equation}
where $D:=\int_0^1 [u^2(u^3)_{hhhh}]^2 \ud h$.
From \eqref{1221_1} and \eqref{1221_2}, for any $T>0$, we obtain
  $$6TE(u(T,\cdot))\leq 6\int_0^T E(u(t,\cdot))\ud t\leq F(u(0,\cdot))-F(u(T,\cdot)),$$
which leads to the algebraic decay
\begin{equation}\label{1221_3}
E(u(T,\cdot))\leq \frac{F(u(0,\cdot))}{6T}, \text{ for any } T>0.
\end{equation}

The free energy $F$ is consistent with the discrete energy $F_N$ defined in \eqref{F_N} and
$E$ was first introduced in the work \cite{She2011}. We call it energy
dissipation rate $E$ due to its physical meaning \eqref{1221_2}, i.e., $E$ gives the rate at which the step free energy $F$ is
dissipated up to a constant. This relation between $E$ and $F$
  is important for proving the positivity, existence and long time behavior of weak solution to
\eqref{PDE}.

On the contrary, if we also had $E(u)\leq c D(u)$ (which does not hold
here), then \eqref{1221_1} would imply
$\frac{\ud E(u)}{\ud t}\leq -cE(u)$, i.e., $E$ is bounded by the
dissipation rate of itself.  This kind of structure would lead to an
exponential decay rate, which is widely used for convergence of weak
solution to its steady state, see e.g., \cite{villani}.  While we do
not have such a classical exponential decay structure, the two related
dissipation structures \eqref{1221_1}, \eqref{1221_2} are good enough
to get an algebraic decay \eqref{1221_3} and obtain the long time
behavior of weak solution; see Section \ref{sec4}.

 We also give a formal observation for the conservation law of $\frac{1}{u}
$ below. It gives the intuition to prove the positivity of weak solution to regularized problem, which leads to the almost everywhere positivity of weak solution to original problem; see Theorem \ref{global_th}.  Multiplying \eqref{PDE} by $\frac{1}{u^2}$ gives
\begin{equation}
  \frac{\ud }{\ud t}\int_0^1 \frac{1}{u} \ud h=\int_0^1 (u^3)_{hhhh} \ud h=0.
\end{equation}
Hence we know $\int_0^1 \frac{1}{u} \ud h$ is a constant of motion for classical solution.
}

One of the main difficulties for PDE \eqref{PDE} is that it becomes
degenerate-parabolic whenever $u$ approaches $0$. As it is not known
if solutions have singularities on the set $\{u=0\}$ or not, we adopt
a regularization method, $\varepsilon$-system, from the work of
\textsc{Bernis and Friedman} \cite{Friedman1990}. First, we define
weak solution in the spirit of \cite{Friedman1990}. Then we study the
$\varepsilon$-system and obtain an unique global weak solution to
$\varepsilon$-system. The positive lower bound of solution to
$\varepsilon$-system is important in the proof of existence of almost
everywhere positive weak solution to PDE \eqref{PDE}.  Observing the
energy dissipation rate $E$ defined in \eqref{E_u} and the
corresponding variational structure, we will make the natural choice
of using $u^3$ as the variable. Yet another difficulty arises since we
do not have lower order estimate for $u^3$ after
regularization. Therefore we need to adopt the \textit{a-priori}
assumption method and verify the \textit{a-priori} assumption by
calculating the positive lower bound of solutions to
$\varepsilon$-system. Finally, we prove the limit of solution to
$\varepsilon$-system is the weak solution to \eqref{PDE}.  {\blue  When it comes to
  establish two energy-dissipation inequalities for the weak solution
  $u$, singularities on set $\{u=0\}$ cause problem too. Hence we also need to take the advantage of the
  $\varepsilon$-system, which allows us avoiding the difficulty due to
  singularities, to obtain the two
  energy-dissipation inequalities.  }

While we prove the existence, the uniqueness of the weak solution is still an open question. Since we consider a degenerate problem not in divergence form, we have not been able to show the uniqueness after the solution touches zero, nor can we obtain any kind of conservation laws rigorously.

One of the closely related
models is the continuum model in DL regime (we set some physical constants to be $1$ for simplicity)
\begin{equation}\label{5n}
 h_t =  \Bigl(-a H(h_x)-\big(\frac{a^2}{h_x}+3h_x\big)h_{xx}\Bigr)_{xx},
 \end{equation}
 which was first proposed by \textsc{Xiang} \cite{Xiang2002}, who
 considered DL type model \eqref{ODE1} with a different chemical
 potential $\mu_i$. More specifically, an additional contribution
 from global step interaction is included besides the local
 terms in the free energy \eqref{Ff},
\begin{equation}\label{Ffn}
  F_N=a\sum_{i=0}^{N-1} f_1\big(\frac{x_{i+1}-x_i}{a}\big)+a^2\sum_{i=0}^{N-1}\sum_{j=0,j\neq i}^{N-1} f_2\big(\frac{x_{j}-x_i}{a}\big),
\end{equation}
with $f_1(r)=\frac{1}{2r^2}$ and $f_2(r)=a^2 \ln\vert r\vert$. While  the free energy $F_N$ is slightly different from that of \cite{Xiang2002},
where the first term $f_1$ is also treated as a global interaction, the formal continuum limit PDE are the same. As argued in
\cite{Xiang2004}, the second term $f_2$ comes from the misfit elastic
interaction between steps, and is hence higher-order in $a$ compared
with the broken bond elastic interaction between steps which
contributes to the first term. Note that \eqref{5n} is a PDE for the
height of the surface as a function of the position and the first two
terms involve the small parameter $a$. We include in the appendix some
alternative forms of the PDE \eqref{PDE}. In particular, when formally ignoring
these terms with small $a$-dependent amplitude, \eqref{5n} becomes
\begin{equation}\label{DL1}
 h_t = -\frac{3}{2}\bigl((h_x)^2\bigr)_{xxx},
\end{equation}
which is parallel to \eqref{eq:hn} in our case.
For the DL type PDE \eqref{DL1}, a fully rigorous understanding is available in \cite{Kohnbook,giga2010}.
\textsc{Kohn} \cite{Kohnbook} pointed out that
a rigorous understanding for the evolution of global solution to ADL type model \eqref{eq:hn} (as well as PDE \eqref{PDE}) is still open because the mobility $\frac{1}{h_x}$ in \eqref{eq:hn} (which equals $1$ in DL model) brings more difficulties.

Recently, \textsc{Dal Maso},
\textsc{Fonseca} and \textsc{Leoni} \cite{Leoni2014} studied the
global weak solution to \eqref{5n} by setting $a=1$ in the equation, i.e.,
\begin{equation}\label{5}
  h_t =  \Bigl(- H(h_x)-\big(3h_x+\frac{1}{h_x}\big)h_{xx}\Bigr)_{xx}.
\end{equation}
The work \cite{Leoni2014} validated \eqref{5} analytically by
verifying the almost everywhere positivity of $h_x$. Moreover,
\textsc{Fonseca}, \textsc{Leoni} and \textsc{Lu} \cite{Leoni2015}
obtained the existence and uniqueness of the weak solution to
\eqref{5}. However, also because the mobility $\frac{1}{h_x}$ (which equals $1$ in DL model) appears when the PDE is rewritten as $h$-equation \eqref{eq:hn}, there is little chance to recast it into {\blue an abstract evolution equation with maximal
monotone operator in reflexive Banach space by choosing other variables, which is the key
to the method in \cite{Leoni2015}. It is very challenged to apply the classical maximal monotone method to a non-reflexive Banach space,}
so we use different
techniques following \textsc{Bernis and Friedman} \cite{Friedman1990}
and the uniqueness is still open.

The remainder of this paper is arranged as follows. After defining the weak solution, Section \ref{sec3}
is devoted to prove the main Theorem \ref{global_th}. In Section
\ref{sec3.1}, we establish the well-posedness of the regularized
$\varepsilon$-system and study its properties. In Section \ref{sec3.2},
we study the existence of global weak solution to PDE \eqref{PDE} and
prove it is positive almost everywhere. In Section \ref{sec3.3}, we
obtain the space-time H\"older continuity of the weak
solution. Section \ref{sec4} considers the long time behavior of weak
solution. The paper ends with Appendix which include a few alternative
formulations of the PDEs based on other physical variables than the
slope.

\section{Global weak solution}\label{sec3}

In this section, we start to prove the global existence and almost everywhere positivity of weak solutions to PDE \eqref{PDE}.
In the following, with standard notations for Sobolev spaces,
denote
\begin{equation}
H^{m}_{\per}([0,1]):=\{u(h)\in H^m(\mathbb{R}); {\blue \,u(h+1)=u(h) \,\,a.e.\,h\in \mathbb{R} } \},
\end{equation}
{\blue and when $m=0$, we denote as $L^2_{\per}([0,1])$.}
We will study the continuum problem \eqref{PDE} in periodic setup.

Although we can prove the measure of $\{(t,x);u(t,x)=0\}$ is zero, we still have no information for it. To avoid the difficulty when $u=0$, we use a regularized method introduced by \textsc{Bernis and Friedman} \cite{Friedman1990}.
Since we do not know the situation in set $\{(t,x);u(t,x)=0\}$,
we need to define a set
\begin{equation}\label{PT}
  P_T:=(0,T)\times(0,1)\backslash \{(t,h);u(t,h)=0\}.
\end{equation}
As a consequence of \eqref{u+} and time-space H\"older
regularity estimates for $u^3$ in Proposition \ref{timeholder}, we know that $P_T$ is an open set and we can
define a distribution on $P_T$.
Recall the definition $E$ in \eqref{E_u}.
First we give the definition of weak solution to PDE \eqref{PDE}.
\begin{defn}
  For any $T>0$, we call a non-negative function $u(t,h)$ with regularities
  \begin{equation}\label{723_1}
    u^3\in \Lty([0,T];\Hper^2([0,1])),\quad u^2(u^3)_{hhhh}\in L^2(P_T),
  \end{equation}
  \begin{equation}\label{723_3}
    u_t\in L^2([0,T];\Lper^2([0,1])),\quad u^3\in C([0,T];\Hper^1([0,1])),
  \end{equation}
   a weak solution to PDE \eqref{PDE} with initial data $u_0$ if
  \begin{enumerate}[(i)]
    \item for any function $\phi\in C^\infty([0,T]\times\mathbb{R})$, which is $1$-periodic with respect to $h$, $u$ satisfies
    \begin{equation}\label{eq01}
      \int_0^T\int_0^1 \phi u_t \ud h \ud t+\int\int_{P_T}\phi{u^2}{(u^3)_{hhhh}}\ud h \ud t =0;
    \end{equation}
    \item the following {\blue first} energy-dissipation inequality holds
    \begin{equation}\label{E01}
      E(u(T,\cdot))+\int\int_{P_T}(u^2(u^3)_{hhhh})^2 \ud h \ud t \leq E(u(0,\cdot)).
    \end{equation}
    {\blue
    \item the following {\blue second} energy-dissipation inequality holds
    \begin{equation}\label{F01}
      F(u(T,\cdot))+6\int_0^TE(u(t,\cdot)) \ud t \leq F(u(0,\cdot)).
    \end{equation}
    }
  \end{enumerate}

\end{defn}

We now state the main result the global existence of weak solution to \eqref{PDE} as follows.
\begin{thm}\label{global_th}
  For any $T>0$, assume initial data $u^3_0\in \Hper^2([0,1])$, $\int_0^1 \frac{1}{u_0} \ud h =m_0<+\infty$ and $u_0\geq0$. Then there exists a global non-negative weak solution to PDE \eqref{PDE} with initial data $u_0$. Besides, we have
  \begin{equation}\label{u+}
    u(t,h)>0, \text{  for }\ale (t,h)\in[0,T]\times[0,1].
  \end{equation}
\end{thm}

We will use an approximation method to obtain the global existence Theorem \ref{global_th}. This method is proposed by \cite{Friedman1990} to study a nonlinear degenerate parabolic equation.

\subsection{Global existence for a regularized problem and some properties}\label{sec3.1}
Consider the following regularized problem in one period $h\in[0,1]$:
\begin{equation}\label{equv}
 \left\{
     \begin{array}{ll}
       \dpstyle{\uvt=-\frac{\uv^4}{\varepsilon+\uv^2}(\uv^3)_{hhhh},} & \text{ for }t\in[0,T],\,h\in[0,1]; \\
       \dpstyle{\uv(0,h)=u_0+\varepsilon^{\frac{1}{3}},} & \text{ for }h\in[0,1].
     \end{array}
   \right.
\end{equation}
 We point out that the added perturbation term is important to the positivity of the global weak solution.

First we give the definition of weak solution to regularized problem \eqref{equv}.
\begin{defn}\label{defnuv}
  For any fixed $\varepsilon>0,\,T>0$, we call a non-negative function $\uv(t,h)$ with regularities
  \begin{equation}
    \uv^3\in \Lty([0,T];{\blue \Hper^2}([0,1])),\quad \frac{\uv^3}{\sqrt{\varepsilon+\uv^2}}(\uv^3)_{hhhh}\in L^2(0,T;{\blue \Lper^2}([0,1])),
  \end{equation}
  \begin{equation}
    \uvt\in L^2([0,T];{\blue \Lper^2}([0,1])),\quad \uv^3\in C([0,T];\Hper^1([0,1])),
  \end{equation}
   weak solution to regularized problem \eqref{equv} if
  \begin{enumerate}[(i)]
    \item for any function $\phi\in C^\infty([0,T]\times[0,1])$, $\uv$ satisfies
    \begin{equation}\label{eqv01}
      \int_0^T\int_0^1 \phi \uvt \ud h \ud t+\int_0^T\int_0^1 \phi\frac{\uv^4}{\varepsilon+\uv^2}(\uv^3)_{hhhh}\ud h \ud t =0;
    \end{equation}
    \item the following {\blue first} energy-dissipation equality holds
    \begin{equation}\label{Ev01}
      E(\uv(T,\cdot))+\int_0^T\int_0^1\Big[\frac{\uv^3}{\sqrt{\varepsilon+\uv^2}}(\uv^3)_{hhhh}\Big]^2 \ud h \ud t =E(\uv(0,\cdot)).
    \end{equation}
    {\blue \item the following second energy-dissipation equality holds
    \begin{equation}\label{Fv01}
      F_\varepsilon(\uv(T,\cdot))+6\int_0^TE(\uv(t,\cdot)) \ud t =F_\varepsilon(\uv(0,\cdot)),
    \end{equation}
    where
    $F_\varepsilon(\uv):=\int_0^1 \varepsilon \ln |\uv| \ud h+F(\uv)$
    is a perturbed version of $F$.  }
  \end{enumerate}
\end{defn}

Now we introduce two lemmas which will be used later.
\begin{lem}\label{lem1}
For any $1$-periodic function $u$, we have the following relation
\begin{equation}\label{equiv1}
    \int_0^1((u^3)_{hh})^2 \ud h=9\int_0^1 u^4 (u_{hh})^2 \ud h.
\end{equation}
\end{lem}
\begin{proof}
Notice that
  \begin{align*}
    ((u^3)_{hh})^2&=[(3u^2u_h)_h]^2=[6uu_h^2+3u^2u_{hh}]^2\\
    &=9u^4u_{hh}^2+36u^2u_h^4+36u^3u_h^2u_{hh}\\
    &=9u^4u_{hh}^2+36u^2u_h^4+12u^3(u_h^3)_h\\
    &=9u^4u_{hh}^2+12(u^3u_h^3)_h.
  \end{align*}
 Integrating from $0$ to $1$, we obtain \eqref{equiv1}.
\end{proof}

\begin{lem}\label{lem2}
  For any function {\blue $u(h)$ such that} $u_{hh}\in L^2([0,1]),$ assume $u$ achieves its minimal value $u_{min}$ at $h^\star$, i.e. $u_{min}=u(h^\star)$. Then we have
  \begin{equation}\label{lem2eq}
     u(h)-u_{min}  \leq \frac{2}{3} \|u_{hh}\|_{L^2([0,1])}\vert h-h^\star \vert^{\frac{3}{2}}, {\blue \text{ for any }h\in[0,1]}.
  \end{equation}
\end{lem}
\begin{proof}
  Since $u_{hh}\in L^2([0,1]),$ $u_h$ is continuous. Hence by $u_{min}=u(h^\star)$, we have $u_h(h^\star)=0$ and
  \begin{equation}
    \vert u_h(h)\vert=\vert \int_{h^\star}^h u_{hh}(s) \ud s \vert \leq \vert h-h^\star \vert^{\frac{1}{2}} \|u_{hh}\|_{L^2([0,1])}, \text{ for any }h\in[0,1].
  \end{equation}
  Hence we have
  \begin{align*}
    \vert u(h)-u_{min}\vert&\leq \int_{h^{\star}}^h \vert s-h^\star \vert^{\frac{1}{2}} \|u_{hh}\|_{L^2([0,1])} \ud s\\
    &\leq \frac{2}{3}\vert h-h^\star \vert^{\frac{3}{2}} \|u_{hh}\|_{L^2([0,1])}.
  \end{align*}
\end{proof}

Next, we study the properties of the regularized problem. {\blue From now on, we denote $C(\|u_0^3\|_{H^2})$ as a constant that only depends on $\|u_0^3\|_{H^2([0,1])}$.} The existence and uniqueness of solution to the regularized problem \eqref{equv} is stated below.
\begin{prop}\label{propuv}
  Assume $u^3_0\in \Hper^2([0,1])$, $\int_0^1 \frac{1}{u_0} \ud h =m_0<+\infty$ and $u_0\geq0$. Then for any $T>0$, there exists $\uv$ being the unique positive weak solution to the regularized system \eqref{equv} and
  \begin{equation}
    \uv^3\in \Lty([0,T];{\blue \Hper^2}([0,1])) \cap C([0,T];{\blue \Hper^1}([0,1]))\nn
\end{equation}
  satisfies the following estimates uniformly in $\varepsilon$
  \begin{equation}\label{es1}
  \begin{aligned}
    &\|\uv^3\|_{\Lty([0,T];H^2([0,1]))}\leq C(\|u_0^3\|_{H^2}),\\
     &\|\frac{\uv^3}{\sqrt{\varepsilon+\uv^2}}(\uv^3)_{hhhh}\|_{L^2([0,T];L^2([0,1]))}\leq C(\|u_0^3\|_{H^2}),
     \end{aligned}
  \end{equation}
  \begin{equation}\label{es2}
    \|\uvt\|_{L^2([0,T];\Lper^2([0,1]))}\leq C(\|u_0^3\|_{H^2}).
  \end{equation}
  Moreover, $\uv$ has the following properties:
  \begin{enumerate}[(i)]
    \item $\uv$ has a positive {\blue lower bound}
  \begin{equation}\label{720_2}
    \uv(t,h) \geq \frac{1}{18^\frac{1}{3}E_0^{\frac{1}{3}}C_ {m_0} }\varepsilon, {\blue \text{ for any }t\in[0,T],\,h\in[0,1],}
  \end{equation}
where $C_{m_0}=\int_0^1 \frac{1}{u_0} \ud h+1$ and $E_0=\int_0^1 \frac{1}{6}[(u_0^3)_{hh}]^2 \ud h$ is the initial energy.
    \item $\uv$ satisfies
     the H\"older continuity properties, i.e.,
  \begin{equation}\label{holder0}
     {\blue \uv^3(t,\cdot)\in C^{\frac{1}{2}}([0,1]), \text{ for any }t\in[0,T].}
   \end{equation}
    \item   For any $\delta>0,$
    \begin{equation}\label{mes}
    {\blue \mu\{(t,h);\uv(t,h)<\delta\}\leq {\blue C_{m_0}}T\delta, }
  \end{equation}
  where $\mu\{A\}$ is the {\blue Lebesgue} measure of set $A$.
  \end{enumerate}

\end{prop}
\begin{proof}
{\blue For a fixed $\varepsilon>0$, in order to get the solution to regularized problem \eqref{equv}, first we need some \textit{a-priori} estimates for $\uv$, the existence of which will discussed later.}
Denote $C_{m_0}:=\int_0^1 \frac{1}{u_0} \ud h+1$, and $u_{min}$ is the minimal value of $\uv$ in $[0,T]\times[0,1]$. {\blue For any $t\in[0,T]$, denote $u_m(t)$ as the minimal value of $\uv(t,h)$ for $h\in[0,1]$.} Assume $\uv$ achieves its minimal value at $t^\star,\,h^\star$, i.e. $u_{min}=\uv(t^\star, h^\star).$  Denote
\begin{equation}
  E_0:=\int_0^1 \frac{1}{6}[(u_0^3)_{hh}]^2 \ud h\leq C(\|u_0^3\|_{H^2}).\nn
\end{equation}
  In Step 1, we first introduce some \textit{a-priori} estimates under  the \textit{a-priori} assumption
\begin{equation}\label{pri_u}
  \uv(t,h) \geq u_{min}\geq \varepsilon^{\frac{4}{3}},\text{  for any }t\in[0,T],\,h\in[0,1].
\end{equation}
In {\blue Step} 2, we prove the {\blue lower bound} of $\uv$ depending on $\varepsilon$, which is the property (i), and verify the \textit{a-priori} assumption \eqref{pri_u}. {\blue After that, the proof for existence of $\uv$ is standard. Here, let us sketch the modified method from \cite{Majda2002}. We can  first modify \eqref{equv} properly using the standard mollifier $J_\delta$ such that the right hand side is locally Lipschitz continuous in Banach space $L^\infty([0,1])$, so that we can
apply the Picard Theorem in abstract Banach space. Hence by \cite[Theorem 3.1]{Majda2002}, it has a unique local solution $u_{\varepsilon\delta}.$ Then by the \textit{a-priori} estimates in Step 1 and Step 2, we can get uniform regularity estimates, extend the maximal existence time for $u_{\varepsilon\delta}$ and finally obtain the limit of $u_{\varepsilon\delta}$, $\uv$, is a weak solution to the regularized problem \eqref{equv}.}
In Step 3, we prove that the solution obtained above is unique. In Step 4, we study the properties (ii) and (iii).
{\blue \begin{rem}
  For the \textit{a-priori} assumption method, to be more transparent, we claim $\uv\geq C_\star \varepsilon$ for any $t\in[0,T]$, where $C_\star=\frac{1}{18^\frac{1}{3}E_0^{\frac{1}{3}}C_ {m_0} }.$ If not, there exists $t_\star\in (0,T)$ such that
  $$\uv(t,h)\geq C_\star \varepsilon, \text{ for any } t\in[0,t_\star],\,h\in[0,1].$$
  Due to the continuity of $\uv,$ there exists $t_{\star\star}\in (t_\star,T)$ such that
  $$\uv(t,h)\geq \varepsilon^{\frac{4}{3}}, \text{ for any }t\in(t_\star,t_{\star\star}),\,h\in[0,1],$$
  and there exists $\tilde{h}\in[0,1],\, \tilde{t}\in (t_\star, t_{\star\star})$ such that
  $$\uv(\tilde{t},\tilde{h})<C_\star \varepsilon.$$
  This is in contradiction with
  $$\uv(t,h)\geq C_\star \varepsilon, \text{ for any }t\in[0,t_{\star\star}),\,h\in[0,1],$$
  which is verified in Step 2.
\end{rem}
}

Step 1. \textit{a-priori} estimates.

First, multiplying \eqref{equv} by $\uv^2$ gives
\begin{equation}
  \frac{1}{3}(\uv^3)_t=-\frac{\uv^6}{\varepsilon+\uv^2}(\uv^3)_{hhhh}.\nn
\end{equation}
Then multiply it by $(\uv^3)_{hhhh}$ and integrate by parts. We have
\begin{equation}\label{dissi1}
  \frac{1}{6}\frac{\ud}{\ud t} \int_0^1 ((\uv^3)_{hh})^2\ud h=-\int_0^1 \frac{\uv^6}{\varepsilon+\uv^2} \big[(\uv^3)_{hhhh}\big]^2 \ud h\leq 0.
\end{equation}
Thus we obtain, for any $T>0$,
\begin{equation}\label{high}
  \|(\uv^3)_{hh}\|_{\Lty([0,T];L^2[0,1])}\leq  \sqrt{6}E_0^{\frac{1}{2}}.
\end{equation}
Moreover, from \eqref{dissi1}, we also have
\begin{equation}\label{temp724}
  \|\frac{\uv^3}{\sqrt{\varepsilon+\uv^2}}(\uv^3)_{hhhh}\|_{L^2([0,T];L^2([0,1]))} \leq  E_0^{\frac{1}{2}}.
\end{equation}

Second, to get the lower order estimate, we need the \textit{a-priori} assumption \eqref{pri_u}.
Multiplying \eqref{equv} by $\frac{\varepsilon+\uv^2}{\uv}$, we have
\begin{equation}\label{add18_1}
\frac{\ud}{\ud t}\int_0^1 \varepsilon \ln\vert\uv\vert+\frac{\uv^2}{2}\ud h=\int_0^1 \big(\frac{\varepsilon}{\uv}+\uv\big)\uvt \ud h=\int_0^1 -((\uv^3)_{hh})^2 \ud h\leq 0,
\end{equation}
which implies {\blue
\begin{align*}
&\int_0^1 \varepsilon \ln\vert\uv(t,\cdot)\vert+\frac{\uv(t,\cdot)^2}{2}\ud h \\
 \leq & \int_0^1 \varepsilon \ln \uv(0)+\frac{\uv(0)^2}{2} \ud h \\
\leq & \int_0^1 \uv(0)^2 \ud h \leq
 C(\|u_0^3\|_{H^2}), \text{ for any }t\in[0,T].
\end{align*}
}
Hence we have
\begin{align*}
  \int_0^1 \frac{\uv{\blue(t,h)}^2}{2}\ud h&\leq -\int_{0,\vert\uv\vert<1}^1 \varepsilon \ln\vert\uv{\blue(t,h)}\vert \ud h + C(\|u_0^3\|_{H^2})\\
  &\leq -\frac{4}{3}\varepsilon\ln \varepsilon +C(\|u_0^3\|_{H^2})\\
  &\leq C(\|u_0^3\|_{H^2}), {\blue \text{ for any }t\in[0,T]}
\end{align*}
where we used the \textit{a-priori} estimate \eqref{pri_u}.
Thus we have, for any $T>0$,
\begin{equation}\label{low}
  \|\uv\|_{\Lty([0,T];L^2[0,1])}\leq  C(\|u_0^3\|_{H^2}).
\end{equation}

Third, from Lemma \ref{lem2}, we have
\begin{equation}\label{720_1}
  \uv{\blue(t,h)}^3-{\blue u_{m}(t)^3}\leq \frac{2}{3}\|(\uv^3)_{hh}{\blue(t,\cdot)}\|_{L^2([0,1])}\vert h-h^\star \vert^{\frac{3}{2}}, \text{ for any }t\in[0,T]\,h\in[0,1].
\end{equation}
Since \eqref{low} gives
\begin{equation}
{\blue u_{m}(t)^3}\leq (\int_0^1 \uv{\blue(t,h)^2} \ud h)^{\frac{3}{2}}\leq C(\|u_0^3\|_{H^2}),\nn \text{for any }t\in[0,T],
\end{equation}
we know
\begin{equation}
  \uv{\blue (t,h)}^3\leq C(\|u_0^3\|_{H^2})+\frac{2\sqrt{6}}{3}E_0^{\frac{1}{2}}\leq C(\|u_0^3\|_{H^2}), \text{ for any }t\in[0,T],\,h\in[0,1],
\end{equation}
where we used \eqref{high} and \eqref{720_1}. Hence we have
\begin{equation}\label{722_3}
\|\uv\|_{\Lty([0,T];\Lty([0,1]))}\leq C(\|u_0^3\|_{H^2}).
\end{equation}
This, together with \eqref{high}, shows that, for any $T>0$,
\begin{equation}\label{high1}
  \|\uv^3\|_{\Lty([0,T];H^2([0,1]))}\leq C(\|u_0^3\|_{H^2}).
\end{equation}
Therefore, \eqref{temp724} and \eqref{high1} yield \eqref{es1}.

On the other hand, from \eqref{dissi1} and \eqref{equv}, we have
\begin{equation}
  \frac{1}{6}\frac{\ud}{\ud t} \int_0^1 ((\uv^3)_{hh})^2\ud h=-\int_0^1  \frac{\uv^2+\varepsilon}{\uv^2}\uvt^2\ud h.\nn
\end{equation}
Hence
\begin{equation}
\int_0^T\int_0^1 \uvt^2\ud h \ud t\leq \int_0^T\int_0^1  \frac{\uv^2+\varepsilon}{\uv^2}\uvt^2\ud h\ud t\leq C(\|u_0^3\|_{H^2}), \nn
\end{equation}
which gives
\begin{equation}\label{time1}
  \|\uvt\|_{L^2([0,T];L^2([0,1]))}\leq C(\|u_0^3\|_{H^2}).
\end{equation}
This, together with \eqref{722_3}, gives that
\begin{equation}\label{time2}
  \|{\blue (\uv^3)_t}\|_{L^2([0,T];L^2([0,1]))}\leq C(\|u_0^3\|_{H^2}).
\end{equation}

In fact, from \eqref{high1} and \eqref{time2}, by \cite[Theorem 4, p.~288]{Evans1998}, we also know
\begin{equation}
\uv^3\in C([0,T];H^1([0,1]))\hookrightarrow C([0,T]\times[0,1]).\nn
\end{equation}

{\blue Moreover, the two dissipation equalities \eqref{Ev01} and \eqref{Fv01} in Definition \ref{defnuv} can be easily obtained from \eqref{dissi1} and \eqref{add18_1} separately.
}

  Step 2. Verify the \textit{a-priori} assumption.

First from \eqref{equv}, we have
\begin{equation}
  \frac{\ud}{\ud t} \int_0^1 \frac{\varepsilon}{3\uv^3}+\frac{1}{\uv} \ud h=0.
\end{equation}
Hence
\begin{equation}\label{w02}
  \int_0^1 \frac{\varepsilon}{3\uv(t,h)^3}+\frac{1}{\uv(t,h)} \ud h\equiv \int_0^1 \frac{\varepsilon}{3(u_0+\varepsilon^{\frac{1}{3}})^3}+\frac{1}{u_0+\varepsilon^{\frac{1}{3}}} \ud h\leq C_ {m_0} , \text{ for any }t\in[0,T].
\end{equation}
Then from \eqref{720_1}, for any {\blue $0<\alpha\leq \frac{1}{2\varepsilon^2}$, $t\in[0,T]$}, we have
  \begin{align*}
    & \frac{\alpha\varepsilon^{3}}{{\blue u_{m}(t)}^3+\frac{2\sqrt{6}E_0^{\frac{1}{2}}}{3}\alpha^{\frac{3}{2}}\varepsilon^{3}} =\int_{h^\star}^{h^\star+\alpha\varepsilon^2}\frac{\varepsilon}{{\blue u_{m}(t)}^3+\frac{2\sqrt{6}E_0^{\frac{1}{2}}}{3}\alpha^{\frac{3}{2}}\varepsilon^{3}}\ud h\\
\leq &\int_0^1 \frac{\varepsilon}{{\blue u_{m}(t)}^3+\frac{2\sqrt{6}E_0^{\frac{1}{2}}}{3}\vert h-h^\star\vert^{\frac{3}{2}}} \ud h \leq \int_0^1 \frac{\varepsilon}{\uv(t,h)^3}\ud h \leq C_ {m_0} .
  \end{align*}
Thus {\blue for any $t\in[0,T]$,} we can directly calculate that, for $\alpha_0=\frac{1}{6E_0 C_ {m_0} ^2},$
\begin{equation}
 {\blue u_{m}(t)}\geq
\Big(\frac{\alpha_0}{C_ {m_0} }-\frac{2\sqrt{6}}{3}E_0^{\frac{1}{2}}\alpha_0^{\frac{3}{2}}\Big)\varepsilon^3 =\frac{1}{18E_0C^3_ {m_0} }\varepsilon^3>> \varepsilon^4,
\end{equation}
and that
\begin{equation}
  u_{min}^3\geq \min_{t\in[0,T]} u_m(t)\geq \frac{1}{18E_0C^3_ {m_0} }\varepsilon^3>> \varepsilon^4,
\end{equation}
for $\varepsilon$ small enough. {\blue Note that for $\varepsilon$ small enough, such $\alpha_0$ can be achieved.} This verifies the \textit{a-priori} assumption and shows that $\uv$ has a positive {\blue lower bound} depending on $\varepsilon$, i.e.,
  \begin{equation*}
    \uv{\blue (t,h)} \geq   \frac{1}{18^\frac{1}{3}E_0^{\frac{1}{3}}C_ {m_0} }\varepsilon, \text{ for any }t\in[0,T],\,h\in[0,1]
  \end{equation*}
which concludes Property (i).

Step 3. Uniqueness of solution to \eqref{equv}.

Assume $u,v$ are two solutions of \eqref{equv}. Then for any fixed $\varepsilon$, from \eqref{720_2}, we know $u,v\geq c_\varepsilon>0$, and we have
\begin{equation}\label{u-v1}
 \frac{1}{3} (u^3-v^3)_t=-\frac{u^6}{u^2+\varepsilon}(u^3)_{hhhh}+\frac{v^6}{v^2+\varepsilon}(v^3)_{hhhh},
\end{equation}
\begin{equation}\label{u-v2}
  (u-v)_t=-\frac{u^4}{u^2+\varepsilon}(u^3)_{hhhh}+\frac{v^4}{v^2+\varepsilon}(v^3)_{hhhh}.
\end{equation}
Let us keep in mind, for any $p\geq0$, $\frac{u^2}{\varepsilon+u^2}u^p$ is increasing {\blue with} respect to $u$, so there exist constants $m,\,M>0$, whose values depend only on $\varepsilon,\,\|u_0^3\|_{H^2([0,1])}$, {\blue $p$} and $ {m_0} $, such that
 \begin{equation}\label{umM}
 m \leq\frac{u^2}{\varepsilon+u^2}u^p\leq M,
 \end{equation}
  and
  \begin{equation}\label{vmM}
  m \leq\frac{v^2}{\varepsilon+v^2}v^p \leq M.
  \end{equation}

  First, multiply \eqref{u-v1} by $(u^3-v^3)_{hhhh}$ and integrate by parts. We have
  \begin{align*}
   &\frac{\ud}{\ud t}  \int_0^1 \frac{1}{6}(u^3-v^3)_{hh}^2 \ud h \\
   =&  \int_0^1 \Big[-\frac{u^6}{u^2+\varepsilon}(u^3)_{hhhh}+\frac{v^6}{v^2+\varepsilon} (u^3)_{hhhh}-\frac{v^6}{v^2+\varepsilon}(u^3)_{hhhh}\\
   & \qquad+\frac{v^6}{v^2+\varepsilon}(v^3)_{hhhh}\Big](u^3-v^3)_{hhhh}\ud h \\
   =&- \int_0^1 \frac{v^6}{v^2+\varepsilon}((u^3-v^3)_{hhhh})^2\ud h + \int_0^1 \Big(\frac{v^6}{v^2+\varepsilon}-\frac{u^6}{u^2+\varepsilon}\Big)(u^3)_{hhhh}(u^3-v^3)_{hhhh}\ud h \\
   =&:R_1+R_2.
  \end{align*}
  For the first term $R_1$, from \eqref{vmM}, we have
  \begin{equation}\label{722R1}
    R_1\leq -m \int_0^1((u^3-v^3)_{hhhh})^2\ud h ,
  \end{equation}
  which will be used to control other terms.

  For the second term $R_2$, notice that
  \begin{equation}\label{722_1}
    \begin{aligned}
      &\Big\|\frac{v^6}{v^2+\varepsilon}-\frac{u^6}{u^2+\varepsilon}\Big\|_{\Lty([0,1])}\\
      =& \Big\|\frac{(u^2+\varepsilon)v^6-(v^2+\varepsilon)u^6}{(v^2+\varepsilon)(u^2+\varepsilon)}\Big\|_{\Lty([0,1])}\\
      =& \Big\|\frac{u^2v^2(v^4-u^4)}{(v^2+\varepsilon)(u^2+\varepsilon)}+\frac{\varepsilon(v^6-u^6)}{(v^2+\varepsilon)(u^2+\varepsilon)}\Big\|_{\Lty([0,1])}\\
      \leq&  C(\|u_0^3\|_{H^2([0,1])} ,\varepsilon, {\blue m_0})\|v-u\|_{\Lty([0,1])},
    \end{aligned}
  \end{equation}
  where we used the upper bound and {\blue lower bound} of $u,v$. Then by Young's inequality and H\"older's inequality, we know
  \begin{equation}\label{722R2}
    \begin{aligned}
      R_2&\leq \frac{m}{4} \int_0^1 ((u^3-v^3)_{hhhh})^2 \ud h  + C \Big\|\frac{v^6}{v^2+\varepsilon}-\frac{u^6}{u^2+\varepsilon}\Big\|_{\Lty([0,1])}^2  \int_0^1 ((u^3)_{hhhh})^2 \ud h \\
      &\leq \frac{m}{4} \int_0^1 ((u^3-v^3)_{hhhh})^2 \ud h  +  C(\|u_0^3\|_{H^2([0,1])} ,\varepsilon,{\blue m_0})\|v-u\|_{\Lty([0,1])}^2,
    \end{aligned}
  \end{equation}
  where we used \eqref{es1} and \eqref{722_1}.
 Combining \eqref{722R1} and \eqref{722R2}, we obtain
   \begin{equation}\label{uniq1}
   \begin{aligned}
    &\frac{\ud}{\ud t}  \int_0^1 \frac{1}{6}(u^3-v^3)_{hh}^2 \ud h \\
\leq &-\frac{3m}{4} \int_0^1 ((u^3-v^3)_{hhhh})^2 \ud h  +  C(\|u_0^3\|_{H^2([0,1])} ,\varepsilon,{\blue m_0})\|v-u\|_{\Lty([0,1])}^2.
    \end{aligned}
  \end{equation}

  Second, multiply \eqref{u-v2} by $u-v$ and integrate by parts. We have
  \begin{align*}
    &\frac{1}{2}\frac{\ud}{\ud t}  \int_0^1 (u-v)^2\ud h  \\
    =&  \int_0^1 \Big[-\frac{u^4}{u^2+\varepsilon}(u^3)_{hhhh}+\frac{v^4}{v^2+\varepsilon} (u^3)_{hhhh}-\frac{v^4}{v^2+\varepsilon}(u^3)_{hhhh}\\
    &\qquad +\frac{v^4}{v^2+\varepsilon}(v^3)_{hhhh}\Big](u-v)\ud h \\
 = &- \int_0^1 \frac{v^4}{v^2+\varepsilon}(u^3-v^3)_{hhhh}(u-v)\ud h + \int_0^1 \Big(\frac{v^4}{v^2+\varepsilon}-\frac{u^4}{u^2+\varepsilon}\Big)(u^3)_{hhhh}(u-v)\ud h \\
   =&:R_3+R_4.
  \end{align*}
  For $R_3$, by H\"older's inequality, we have
  \begin{equation}\label{722R3}
    R_3\leq \frac{m}{4} \int_0^1 ((u^3-v^3)_{hhhh})^2 \ud h  +C(\|u_0^3\|_{H^2([0,1])} ,{\blue m_0})\|v-u\|_{L^2([0,1])}^2,
  \end{equation}
where we used \eqref{vmM}.
To estimate $R_4$, notice that
  \begin{equation}\label{722_2}
    \begin{aligned}
      &\Big\|\frac{v^4}{v^2+\varepsilon}-\frac{u^4}{u^2+\varepsilon}\Big\|_{\Lty([0,1])}\\
      =& \Big\|\frac{(u^2+\varepsilon)v^4-(v^2+\varepsilon)u^4}{(v^2+\varepsilon)(u^2+\varepsilon)}\Big\|_{\Lty([0,1])}\\
      =& \Big\|\frac{u^2v^2(v^2-u^2)}{(v^2+\varepsilon)(u^2+\varepsilon)}+\frac{\varepsilon(v^4-u^4)}{(v^2+\varepsilon)(u^2+\varepsilon)}\Big\|_{\Lty([0,1])}\\
      \leq&  C(\|u_0^3\|_{H^2([0,1])} ,\varepsilon, {\blue m_0})\|v-u\|_{\Lty([0,1])}.
    \end{aligned}
  \end{equation}
Hence, we have
  \begin{equation}\label{722R4}
    \begin{aligned}
      R_4&\leq C \int_0^1 (u-v)^2 \ud h  + C \Big\|\frac{v^4}{v^2+\varepsilon}-\frac{u^4}{u^2+\varepsilon}\Big\|_{\Lty([0,1])}^2  \int_0^1 ((u^3)_{hhhh})^2 \ud h \\
      &\leq C \int_0^1 (u-v)^2 \ud h  +  C(\|u_0^3\|_{H^2([0,1])} ,\varepsilon, {\blue m_0})\|v-u\|_{\Lty([0,1])}^2.
    \end{aligned}
  \end{equation}
  Therefore, combining \eqref{722R3} and \eqref{722R4}, we {\blue obtain}
  \begin{equation}\label{uniq2}
  \begin{aligned}
    &\frac{1}{2}\frac{\ud}{\ud t}  \int_0^1 (u-v)^2\ud h  \\
    \leq& \frac{m}{4} \int_0^1 ((u^3-v^3)_{hhhh})^2 \ud h  +C(\|u_0^3\|_{H^2([0,1])},{\blue m_0} ,\varepsilon)\|v-u\|_{\Lty([0,1])}^2.
    \end{aligned}
  \end{equation}

Finally, \eqref{uniq1} and \eqref{uniq2} show that
  \begin{equation}\label{temp722_1}
  \begin{aligned}
    &\frac{\ud}{\ud t}\Big[  \int_0^1 (u-v)^2\ud h  +  \int_0^1 (u^3-v^3)_{hh}^2 \ud h \Big]\\
    \leq& C(\|u_0^3\|_{H^2([0,1])} ,\varepsilon,{\blue m_0})\|v-u\|_{\Lty([0,1])}^2.
    \end{aligned}
  \end{equation}

  In remains to show the right-hand-side of \eqref{temp722_1} is controlled by $\int_0^1 (u-v)^2\ud h  +  \int_0^1 (u^3-v^3)_{hh}^2 \ud h .$
From \eqref{720_2}, we have
\begin{equation}
c_\varepsilon\vert u-v \vert\leq \vert u-v \vert(u^2+v^2+uv)=\vert u^3-v^3 \vert. \nn
\end{equation}
Thus
\begin{align*}
  \|v-u\|_{\Lty([0,1])}^2 &  \leq  c_\varepsilon\|v^3-u^3\|_{\Lty([0,1])}^2 \\
                          &   \leq c_\varepsilon\|v^3-u^3\|_{H^2([0,1])}^2\\
                          & \leq c_\varepsilon\big(\|v^3-u^3\|_{L^2([0,1])}^2+\|(v^3-u^3)_{hh}\|_{L^2([0,1])}^2\big)\\
                          & \leq c_\varepsilon\big(\|v-u\|_{L^2([0,1])}^2+\|(v^3-u^3)_{hh}\|_{L^2([0,1])}^2\big).
\end{align*}
This, together with \eqref{temp722_1}, gives
  \begin{equation}
  \begin{aligned}
    &\frac{\ud}{\ud t}\Big[  \int_0^1 (u-v)^2\ud h  +  \int_0^1 (u^3-v^3)_{hh}^2 \ud h \Big]\\
    \leq &C(\|u_0^3\|_{H^2([0,1])}, m_0,\varepsilon) \Big[  \int_0^1 (u-v)^2\ud h  +  \int_0^1 (u^3-v^3)_{hh}^2 \ud h \Big].
    \end{aligned}
  \end{equation}
Hence if $u(0)=v(0)$, Gr\"onwall's inequality implies $u=v$.

Step 4. The properties (ii) and (iii).

To obtain (ii), denote $w=\uv^3$. From \eqref{es1}, we know $w\in \Lty(0,T;H^2([0,1]))$. Since $H^2([0,1])\hookrightarrow C^{1,\frac{1}{2}}([0,1])$, we can get \eqref{holder0} directly.

To obtain (iii), for any $\delta>0$, \eqref{w02} also gives that
\begin{equation}
  \mu\{(t,h);\uv<\delta\}\frac{1}{\delta}\leq \int_0^T\int_0^1 \frac{\varepsilon}{3\uv^3}+\frac{1}{\uv} \ud h {\blue \ud t} {\blue \leq C_{m_0}}T,\nn
\end{equation}
which concludes \eqref{mes}.

This completes the proof of Proposition \ref{propuv}.
\end{proof}

\subsection{Global existence of weak solution to PDE \eqref{PDE}}\label{sec3.2}

After those preparations for regularized system, we can start to {\blue prove} the global weak solution of \eqref{PDE}.
\begin{proof}[Proof of Theorem \ref{global_th}]
In Step 1 and Step 2, we will first prove the regularized solution $\uv$ obtained in Proposition \ref{propuv} converge to $u$, and $u$ is positive almost everywhere. Then in Step 3 and Step 4, we prove this $u$ is the weak solution to PDE \eqref{PDE} by verifying condition \eqref{eq01} and \eqref{E01}.

Step 1. Convergence of $\uv$.

Assume $\uv$ is the weak solution to \eqref{equv}.
From \eqref{es1} and \eqref{es2}, we have
\begin{equation}
  \|(\uv^3)_t\|_{L^2([0,T];\Lper^2([0,1]))}\leq C(\|u_0^3\|_{H^2} ).\nn
\end{equation}
Therefore, as $\varepsilon\rightarrow 0,$  we can use Lions-Aubin's compactness lemma for $\uv^3$ to show that there exist a subsequence of $\uv$ (still denoted by $\uv$) and $u$ such that
\begin{equation}\label{strong}
  \uv^3\rightarrow u^3, \text{ in }\Lty([0,T];\Hper^1([0,1])),
\end{equation}
which gives
\begin{equation}\label{strong1}
  \uv \rightarrow u, \quad \ale t\in[0,T],\,h\in[0,1].
\end{equation}
Again from \eqref{es1} and \eqref{es2}, we have
\begin{equation}\label{weak1}
  \uv^3 {\stackrel{\star}{ \rightharpoonup}} u^3 \quad \text{in }\Lty([0,T];\Hper^2([0,1])),
\end{equation}
and
\begin{equation}\label{weak2}
  \uvt \rightharpoonup u_t \quad \text{in }L^2([0,T];\Lper^2([0,1])),
\end{equation}
which imply that
\begin{equation}\label{weak3}
  u^3\in \Lty([0,T];\Hper^2([0,1])),\quad  u_t \in L^2([0,T];\Lper^2([0,1])).
\end{equation}
In fact, by \cite[Theorem 4, p.~288]{Evans1998}, we also know
\begin{equation}
u^3\in C([0,T];\Hper^1([0,1]))\hookrightarrow C([0,T]\times[0,1]).\nn
\end{equation}

Step 2. Positivity of $u$.

From \eqref{strong1}, we know, {\blue up to a set of measure zero,}
\begin{equation}
  \{(t,h);u(t,h)=0\}\subset \bigcap_{n=1}^{\infty}\{(t,h);\uv <\frac{1}{n}\}.\nn
\end{equation}
Hence by \eqref{mes} in Proposition \ref{propuv}, we have
\begin{equation}
  \mu\{(t,h);u(t,h)=0\}=\lim_{n\rightarrow 0} \mu\{(t,h);\uv<\frac{1}{n}\}=0,\nn
\end{equation}
which concludes $u$ is positive almost everywhere.

Step 3. $u$ is a weak solution of \eqref{PDE} satisfying \eqref{eq01}.

Recall $\uv$ is the weak solution of \eqref{equv} satisfying \eqref{eqv01}. We want to pass the limit for $\uv$ in \eqref{eqv01} as $\varepsilon \rightarrow 0.$ From \eqref{weak2}, the first term in \eqref{eqv01} becomes
\begin{equation}
  \int_0^T\int_0^1 \phi \uvt \ud h \ud t \rightarrow  \int_0^T\int_0^1 \phi u_t \ud h \ud t .
\end{equation}
The limit of the second term in \eqref{eqv01} is given by the following claim:
\begin{clm}\label{claim3.5}
  For $P_T$ defined in \eqref{PT}, {\blue for any function $\phi\in C^\infty([0,T]\times[0,1]),$} we have
  \begin{equation}\label{claim}
     \int_0^T\int_0^1 \phi \frac{\uv^4}{\varepsilon+\uv^2}(\uv^3)_{hhhh} \ud h \ud t \rightarrow  \int\int_{P_T} \phi u^2 (u^3)_{hhhh} \ud h \ud t,
  \end{equation}
  as $\varepsilon\rightarrow 0.$
\end{clm}
\begin{proof}[Proof of claim]
First, for any fixed $\delta>0$, from \eqref{strong}, we know there exist a constant $K_1>0$ large enough and a subsequence $\uvk$  such that
\begin{equation}\label{del1}
  \|\uvk-u\|_{\Lty([0,T]\times[0,1])}\leq \frac{\delta}{2},\text{  for }k>K_1.
\end{equation}
Denote
\begin{align}
  & \Dld:=\{h\in[0,1];\,0\leq u(t,h)\leq \delta\}, \nn \\
  & \Dgd:=\{h\in[0,1];\,u(t,h)> \delta\}. \nn
\end{align}
The left-hand-side of \eqref{claim} becomes
\begin{align*}
  &\int_0^T\int_0^1 \phi \frac{\uvk^4}{\vark+\uvk^2}(\uvk^3)_{hhhh} \ud h \ud t\\
  =&\int_0^T\int_{\Dld} \phi \frac{\uvk^4}{\vark+\uvk^2}(\uvk^3)_{hhhh} \ud h \ud t+\int_0^T\int_{\Dgd} \phi \frac{\uvk^4}{\vark+\uvk^2}(\uvk^3)_{hhhh} \ud h \ud t\\
  =:& I_1+I_2.
\end{align*}
Then we estimate $I_1$ and $I_2$ separately.

For $I_1$, from \eqref{del1}, we have
\begin{equation}
  \vert \uvk(t,h) \vert\leq \frac{3\delta}{2}, \text{ for }t\in[0,T],\,h\in \Dld.
\end{equation}
Hence by H\"older's inequality, we know
\begin{align}\label{11I1}
 I_1\leq& \Big[\int_0^T\int_{\Dld} \Big(\phi \frac{\uvk}{\sqrt{\vark+\uvk^2}} \Big)^2 \ud h \ud t \Big]^{\frac{1}{2}}\\
 &\cdot \Big[\int_0^T\int_{\Dld} \Big( \frac{\uvk^3}{\sqrt{\vark+\uvk^2}}(\uvk^3)_{hhhh} \Big)^2 \ud h \ud t \Big]^{\frac{1}{2}}\nn\\
 \leq& {\blue C(\|u_0^3\|_{H^2})\|\phi\|_{L^\infty([0,T]\times[0,1])} \Big(\mu\big\{(t,h);\,\vert\uvk\vert\leq\frac{3\delta}{2}\big\}\Big)^{\frac{1}{2}} }\nn\\
  \leq& {\blue C(\|u_0^3\|_{H^2})T^{\frac{1}{2}}\delta^{\frac{1}{2}} }.\nn
\end{align}
Here we used \eqref{es1} in the second inequality and \eqref{mes} in the last inequality.

Now we turn to estimate $I_2$.
Denote
\begin{equation}
B_\delta:=\bigcup_{t\in[0,T]} \{t\}\times\Dgd.\nn
\end{equation}
From \eqref{del1}, we know
\begin{equation}
  \uvk{\blue (t,h)}>\frac{\delta}{2}, {\blue \text{  for }(t,h)\in B_{\delta}.}
\end{equation}
This, {\blue combined} with \eqref{es1}, shows that
\begin{equation}\label{temp711}
\begin{aligned}
  &\frac{\big(\frac{\delta}{2}\big)^6}{\vark+\big(\frac{\delta}{2}\big)^2}\int\int_{B_\delta}((\uvk^3)_{hhhh})^2 \ud h \ud t \\
  \leq&  \int_0^T\int_0^1\frac{\uvk^6}{\vark+\uvk^2}((\uvk^3)_{hhhh})^2 \ud h \ud t\leq C(\|u_0^3\|_{\Hper^2([0,1])}).
  \end{aligned}
\end{equation}
From \eqref{temp711} and \eqref{strong1}, there exists a subsequence of $\uvk$ (still denote as $\uvk$) such that
$$(\uvk^3)_{hhhh} \rightharpoonup (u^3)_{hhhh}, \text{ in }L^2(B_{\delta}).$$
Hence, together with \eqref{strong1}, we have
\begin{equation}\label{11I2}
  I_2=\int\int_{B_{\delta}}\phi \frac{\uvk^4}{\vark+\uvk^2}(\uvk^3)_{hhhh} \ud h \ud t \rightarrow \int\int_{B_{\delta}}\phi u^2(u^3)_{hhhh} \ud h \ud t.
\end{equation}
Combining \eqref{11I1} and \eqref{11I2}, we know there exists $K>K_1$ large enough such that for $k>K,$
$$\Big\vert \int_0^T\int_0^1 \phi \frac{\uvk^4}{\vark+\uvk^2}(\uvk^3)_{hhhh} \ud h \ud t-\int\int_{B_\delta} \phi u^2(u^3)_{hhhh} \ud h \ud t \Big\vert\leq {\blue C(\|u_0^3\|_{H^2})T^{\frac{1}{2}}\delta^{\frac{1}{2}} },$$
which implies that
$$\lim_{\delta\rightarrow 0^+}\lim_{k\rightarrow\infty}\Big[ \int_0^T\int_0^1 \phi \frac{\uvk^4}{\vark+\uvk^2}(\uvk^3)_{hhhh} \ud h \ud t-\int\int_{B_\delta} \phi u^2(u^3)_{hhhh} \ud h \ud t \Big]=0.$$
{\blue For any $\ell\geq 1$, assume the sequence $\delta_\ell\rightarrow 0$. Thus we can choose a sequence $\varepsilon_{\ell k}\rightarrow +\infty.$ Then by the diagonal rule, we have
$$\delta_\ell\rightarrow 0,\quad \varepsilon_{\ell\ell}\rightarrow +\infty,$$
as $\ell$ tends to $+\infty$.}
Notice $$P_T=\bigcup_{\delta>0}B_{\delta}.$$
We have
{ \blue \begin{align*}
&\lim_{\ell\rightarrow\infty} \int_0^T\int_0^1 \phi \frac{u_{\varepsilon_{\ell\ell}}^4}{\varepsilon_{\ell\ell}+u_{\varepsilon_{\ell\ell}}^2}(u_{\varepsilon_{\ell\ell}}^3)_{hhhh} \ud h \ud t  \\
=&\lim_{\ell\rightarrow\infty}\int\int_{B_{\delta_\ell}} \phi u^2(u^3)_{hhhh} \ud h \ud t\\
=&\int\int_{P_T} \phi u^2(u^3)_{hhhh} \ud h \ud t.
\end{align*}
}
This completes the proof of the claim.
\end{proof}
Hence the function $u$ obtained in Step 1 satisfies weak solution form \eqref{eq01}. It remains to verify \eqref{E01} in Step 4.

Step 4. Energy-dissipation inequality \eqref{E01} and \eqref{F01}.

First recall the regularized solution $\uv$ satisfies the Energy-dissipation equality \eqref{Ev01}, i.e.,
\begin{equation*}
      E(\uv(\cdot,T))+\int_0^T\int_0^1\Big[\frac{\uv^3}{\sqrt{\varepsilon+\uv^2}}(\uv^3)_{hhhh}\Big]^2 \ud h \ud t =E(\uv(\cdot,0)).
    \end{equation*}
From the {\blue Claim \ref{claim3.5},} we have
$$\frac{\uv^4}{\varepsilon+\uv^2}(\uv^3)_{hhhh} \rightharpoonup  u^2 (u^3)_{hhhh}, \text{ in }P_T.$$
Then by the lower semi-continuity of norm, we know
\begin{equation}\label{ss1}
\begin{aligned}
  \int\int_{P_T}(u^2(u^3)_{hhhh})^2\ud h\ud t\leq& \liminf_{\varepsilon\rightarrow 0} \int\int_{P_T} \Big[\frac{\uv^4}{\varepsilon+\uv^2}(\uv^3)_{hhhh} \Big]^2 \ud h\ud t\\
  \leq&\liminf_{\varepsilon\rightarrow 0} \int\int_{P_T} \Big[\frac{\uv^3}{\sqrt{\varepsilon+\uv^2}}(\uv^3)_{hhhh} \Big]^2 \ud h\ud t.
  \end{aligned}
\end{equation}
Also from \eqref{es1}, we have
\begin{equation}\label{ss2}
  E(u(t,\cdot))\leq \liminf_{\varepsilon\rightarrow 0} E(\uv(t,\cdot)), \text{ for } t\in[0,T].
\end{equation}
Combining \eqref{Ev01}, \eqref{ss1} and \eqref{ss2}, we obtain
$$E(u(T,\cdot))+\int\int_{P_T}(u^2(u^3)_{hhhh})^2 \ud h \ud t \leq E(u(0,\cdot)).$$

{\blue
Second, recall the regularized solution $\uv$ satisfies the Energy-dissipation equality \eqref{Fv01}, i.e.,
\begin{equation*}
      F_\varepsilon(\uv(T,\cdot))+6\int_0^TE(\uv(t,\cdot)) \ud t =F_\varepsilon(\uv(0,\cdot)).
\end{equation*}
From \eqref{es1} and the lower semi-continuity of norm, we know
\begin{equation}\label{add18_2}
 \begin{aligned}
  &\int_0^TE(u(t,\cdot))\ud t\leq \liminf_{\varepsilon\rightarrow 0} \int_0^T E(\uv(t,\cdot))\ud t,\\
  &F(u(t,\cdot))\leq \liminf_{\varepsilon\rightarrow 0} F(\uv(t,\cdot)), \text{ for any }t\in[0,T].
  \end{aligned}
\end{equation}
For the first term in $F_\varepsilon$, for any $t\in[0,T]$, from \eqref{es1} and \eqref{720_2}, we have
$$\varepsilon \int_0^1 |\ln \uv|\ud h\leq C(|\ln \varepsilon| +1)\varepsilon \rightarrow 0, $$
as $\varepsilon$ tends to $0$. This, together with \eqref{add18_2}, implies
$$F(u(T,\cdot))+6\int_0^TE(u(t,\cdot)) \ud t \leq F(u(0,\cdot)).$$
Hence we complete the proof of Theorem \ref{global_th}.
}
\end{proof}

\subsection{Time H\"older regularity of weak solution}\label{sec3.3}
In the following, we study the time-space H\"older regularity of weak solution to PDE \eqref{PDE}.
\begin{prop}\label{timeholder}
 Assume the initial data $u_0$ satisfies the same assumption {\blue as} in Theorem \ref{global_th}. Let $u$ be a non-negative weak solution to PDE \eqref{PDE} with initial data $u_0$. Then $u^3$ has time-space H\"older continuity in the following sense: for any $t_1,t_2\in[0,T],$ $u^3$ satisfies
  \begin{equation}\label{holder}
    \vert u^3(t_1,h)-u^3(t_2,h) \vert \leq C(\|u_0^3\|_{H^2})\big\vert t_2-t_1 \big\vert^{\frac{1}{4}}, \text{ for any }h\in[0,1];
   \end{equation}
  and
     \begin{equation}\label{holderspace}
    {\blue \uv^3(t,\cdot)\in C^{\frac{1}{2}}([0,1]), \text{ for any }t\in[0,T].}
   \end{equation}
\end{prop}
\begin{proof}
First, \eqref{holderspace} is a direct consequence of $u^3\in \Lty([0,T];H^2([0,1]))$ and the embedding $H^2([0,1])\hookrightarrow C^{1,\frac{1}{2}}([0,1])$.

 Second, define two cut-off functions as \cite[Lemma B.1]{jinhuan}. For any $t_1,t_2\in[0,T]$, $t_1<t_2,$ we construct $\bdt=\int_{-\infty}^t b_\delta'(t)dt,$ with $b_\delta'(t)$ satisfying
\begin{equation}
  b_\delta'(t)=\left\{
                 \begin{array}{ll}
                   \frac{1}{\delta}, & \vert t-t_2\vert<\delta, \\
                   -\frac{1}{\delta}, & \vert t-t_1\vert<\delta, \\
                   0, & \hbox{otherwise,}
                 \end{array}
               \right.
\end{equation}
where the constant $\delta$ satisfies $0<\delta<\frac{\vert t_2-t_1\vert}{2}.$ Then it is obvious that $\bdt$ is Lipschitz continuous and satisfies $\vert \bdt\vert\leq2.$

For any $h_0\in(0,1),$ we construct an auxiliary function
\begin{equation}
  a(h)=a_0\Big(\frac{K(h-h_0)}{\vert t_2-t_1 \vert^\alpha}\Big),
\end{equation}
where $0<\alpha<1,\,K>0$ are constants determined later and $a_0(h)\in C_0^\infty(\mathbb{R})$ is defined by
\begin{equation*}
  a_0(h)=\left\{
           \begin{array}{ll}
             1, & -\frac{1}{2}\leq h \leq \frac{1}{2}, \\
             0, & \vert {\blue h} \vert \geq 1.
           \end{array}
         \right.
\end{equation*}
Hence we have
\begin{equation*}
  a(h)=\left\{
           \begin{array}{ll}
             1, & \vert h-h_0 \vert \leq \frac{1}{2K}\vert t_2-t_1 \vert^\alpha, \\
             0, & \vert h-h_0 \vert \geq \frac{1}{K}\vert t_2-t_1 \vert^\alpha.
           \end{array}
         \right.
\end{equation*}
In the following, $C$ is a general constant depending only on $\|u_0^3\|_{H^2([0,1])}$.

Third, since \eqref{723_1} implies $u^3\in \Lty([0,T];H^2([0,1]))\hookrightarrow \Lty([0,T];W^{1,\infty}([0,1]))$, we know for any $y\in\mathbb{R}$, $t\in[0,T]$,
\begin{equation}\label{u3h}
  \vert u^3(t,h_0+y)-u^3(t,h_0) \vert\leq C\vert y \vert.
\end{equation}
Then we have
\begin{lem}\label{lem723}
  Let function $u^3\in \Lty([0,T]; H^2([0,1])).$ Then for almost everywhere $h_0\in[0,1],$ $t_1,t_2\in[0,T],\,t_1<t_2$, it holds
\begin{align}\label{add17_1}
  &\vert u^3(t_2 ,h_0)-u^3(t_1 ,h_0) \vert\\
  \leq& C(\|u_0^3\|_{H^2([0,1])} ,T) \Big( \int_0^T\int_0^1 u^3(t,h)a(h)b_\delta'(t) \ud h\ud t \vert t_2-t_1 \vert^{-\alpha}+ \vert t_2-t_1 \vert^{\alpha}\Big).\nn
\end{align}
\end{lem}
\begin{proof} The proof of Lemma \ref{lem723} is the same as Lemma B.2 in \cite{jinhuan} {\blue except we proceed on $u^3$  instead of $u(t,\cdot)\in C^{\frac{1}{2}}([0,1])$ in \cite[Lemma B.2]{jinhuan}. We just sketch the idea here. First calculate the inner product of $u^3(t,h)$ and $a(h)b_\delta'(t).$ Then by the definition of $b_\delta'(t)$ and \eqref{u3h}, we have
\begin{multline*}
\int_0^T\int_0^1 u^3(t,h)a(h)b_\delta'(t)\ud h \ud t\\
\geq \frac{1}{\delta}\int_{-\delta}^{\delta} \int_{-\frac{1}{K}|t_2-t_1|^\alpha}^{\frac{1}{K}|t_2-t_1|^\alpha} a(h_0+y)\big(u^3(t_2+\tau,h_0)-u^3(t_1+\tau,h_0)\big)\ud y \ud \tau-C(t_2-t_1)^{\frac{3\alpha}{2}}.
\end{multline*}
 Notice the definition of $a(h)$ and the Lebesgue differentiation theorem. Let $\delta$ tend to $0$, and thus we obtain \eqref{add17_1}.
}
\end{proof}

Finally, since the solution $u$ satisfies \eqref{eq01}, for any $\phi_i\in C^\infty([0,T]\times[0,1])$, $u$ satisfies
\begin{equation}\label{726_1}
      \int_0^T\int_0^1 \phi_i u_t \ud h \ud t+\int\int_{P_T}\phi_i{u^2}{(u^3)_{hhhh}}\ud h \ud t =0.
\end{equation}
We can take $\phi_i$ such that $\phi_i\rightarrow u^2a(h)\bdt$ in $L^2([0,T];L^2([0,1]))$ as $i\rightarrow \infty.$
Hence from \eqref{723_1} and \eqref{723_3}, we can take a limit in \eqref{726_1} to obtain
\begin{equation*}
      \int_0^T\int_0^1  \big(\frac{1}{3}u^3\big)_t a(h)\bdt \ud h \ud t+\int\int_{P_T}{u^4}{(u^3)_{hhhh}}a(h)\bdt \ud h \ud t =0.
\end{equation*}
Therefore, using \eqref{723_1}, we have
\begin{align*}
    &\Big\vert  \int_0^T\int_0^1  \big(\frac{1}{3}u^3\big)_t a(h)\bdt \ud h \ud t\Big\vert\\
\leq& \|{u^4}{(u^3)_{hhhh}}\|_{L^2(P_T)}  \|a(h)\bdt\|_{L^2([0,T];L^2([0,1]))}\\
\leq& C \|a(h)\bdt\|_{L^2([0,T];L^2([0,1]))}.
\end{align*}

Noticing the denifitions of $a(h)$ and $\bdt$, we can calculate that
\begin{equation}\label{724_1}
\begin{aligned}
  &\Big\vert \int_0^T\int_0^1 \frac{1}{3}u^3 a(h) b_\delta'(t) \ud h \ud t \Big\vert =  \Big\vert \int_0^T\int_0^1 \big(\frac{1}{3}u^3\big)_t a(h) \bdt \ud h \ud t \Big\vert\\
\leq & C \|a(h)\bdt\|_{L^2([0,T];L^2([0,1]))}=\Big( \int_{h_0-\frac{1}{K}\vert t_2-t_1 \vert^\alpha}^{h_0+\frac{1}{K}\vert t_2-t_1 \vert^\alpha}a(h)^2 \ud h\Big)^{\frac{1}{2}} \Big( \int_0^T b_\delta^2(t)\ud t  \Big)^{\frac{1}{2}}\\
\leq & C\vert t_2-t_1+2\delta \vert^{\frac{1}{2}} \leq C\vert t_2-t_1 \vert^{\frac{1}{2}},
\end{aligned}
\end{equation}
where we used $\delta<\frac{\vert t_2-t_1\vert}{2}.$

Therefore, \eqref{724_1} and Lemma \ref{lem723} show that
\begin{align*}
  &\vert u^3(t_2 ,h_0)-u^3(t_1 ,h_0) \vert\\
\leq& C(\|u_0^3\|_{H^2([0,1])} ,T) \Big( \int_0^T\int_0^1 u^3(t,h)a(h)b_\delta'(t) \ud h\ud t \vert t_2-t_1 \vert^{-\alpha}+ \vert t_2-t_1 \vert^{\alpha}\Big)\\
\leq & C(\|u_0^3\|_{H^2([0,1])} ,T) \Big( \vert t_2-t_1 \vert^{\frac{1}{2}-\alpha}+ \vert t_2-t_1 \vert^{\alpha} \Big),
\end{align*}
{\blue for almost everywhere $h_0\in[0,1],$ $t_1,t_2\in[0,T],\,t_1<t_2$.}
Taking $\alpha=\frac{1}{4},$ we conclude \eqref{holder} and complete the proof of Proposition \ref{timeholder}.
\end{proof}

\section{Long time behavior of weak solution}\label{sec4}
After establishing the global-in-time weak solution, we want to study how the solution will behavior as time goes to infinity. In our periodic setup, it turns out to be a constant solution of PDE \eqref{PDE}.
\begin{thm}
Under the same assumptions of Theorem \ref{global_th}, for every weak solution $u$ obtained in Theorem \ref{global_th}, there exists a constant $u^\star$ such that,
   as time $t\rightarrow +\infty,$ $u$ converges to $u^\star$ in the sense
\begin{equation}
 \|u^3-(u^\star)^3\|_{ H^1([0,1])}\rightarrow 0, \text{  as }t\rightarrow +\infty,
\end{equation}
and
  \begin{equation}\label{810_3}
   \|u-u^\star\|_{\Lty([0,1])}\rightarrow 0, \text{  as }t\rightarrow +\infty.
 \end{equation}
\end{thm}
\begin{proof}
Step 1. Limit of free energy $E(u(t))$.

 For any $T>0$, {\blue from the second energy-dissipation inequality \eqref{F01}, we have
 \begin{equation}\label{add18_3}
 \int_0^1  u(T)^2  \ud h+12\int_0^T E(u(t,\cdot)) \ud t\leq \int_0^1  u_0^2  \ud h.
 \end{equation}
 }
By \eqref{E01}, we know $E(u(t))$ is decreasing with respect to $t$. Then \eqref{add18_3} implies
\begin{equation}
  12T E(u(T))\leq 12\int_0^T E(u(t,\cdot)) \ud t\leq \int_0^1  u_0^2  \ud h-\int_0^1  u(T)^2  \ud h\leq  \int_0^1  u_0^2  \ud h.
\end{equation}
Hence we have
\begin{equation}\label{802_2}
  E(u(t,\cdot))\leq \frac{c}{t}\rightarrow 0, \text{  for any }t\geq0,
\end{equation}
which shows that $E(u(t))$ converges to its minimum $0$ as $t\rightarrow +\infty$.

On the other hand,
denote $w:=u^3$, and
$$E(w)=\int_0^1 ((u^3)_{hh})^2 \ud h=\int_0^1 (w_{hh})^2 \ud h.$$
Since $E(w)$ is strictly convex in $\dot{H}^2$ and $E(w)\rightarrow +\infty$ when $\|w\|_{\dot{H}^2}\rightarrow +\infty$,
hence $E(w)$ achieves its minimum $0$ at unique critical point $w^\star$ in $\dot{H}^2$.
Notice $w$ is periodic so $w^\star\equiv constant$.

Step 2. Convergence of solution to its unique stationary solution.

Assume $u^3\in \Lty([0,\infty);H^2([0,1]))$ is a solution of \eqref{PDE}. Notice $H^2([0,1])\hookrightarrow H^1([0,1])$ compactly.
Then for any sequence $t_n\rightarrow +\infty,$ there exists a subsequence $t_{n_k}$ and $f^\star(h) \text{  in } H^1([0,1])$ such that
\begin{equation}\label{802f}
  u^3(t_{n_k},{\blue\cdot})\rightarrow f^\star({\blue\cdot}), \text{  in } H^1([0,1]) \text{ as }t_{n_k}\rightarrow +\infty.
\end{equation}
From \eqref{802_2} and the uniqueness of critical point, we have
$$\int_0^1 ((u({\blue t,\cdot})^3)_{hh})^2 \ud h \rightarrow \int_0^1 ((w^\star)_{hh})^2 \ud h=0, \text{ as }t\rightarrow +\infty.$$
Hence
\begin{equation}
  u^3(t,{\blue\cdot})\rightarrow w^\star \text{  in } \dot{H}^2([0,1]), \text{ as }t\rightarrow +\infty.
\end{equation}
Since $u$ is periodic, we have poincare inequality for $(u^3)_h$ and
 \begin{equation}
  u^3(t,\cdot)\rightarrow w^\star \text{  in } \dot{H}^1([0,1]), \text{ as }t\rightarrow +\infty.
\end{equation}
This, together with \eqref{802f}, gives
$$f^\star_h\equiv 0,$$
which implies $f^\star$ is also a constant.

Next we state the constant is unique. Denote $u^\star=(f^\star)^{\frac{1}{3}}$. From \eqref{802f} we know
 \begin{equation}\label{810}
   \|u({\blue t_{n_k},\cdot})^3-(u^\star)^3\|_{\Lty([0,1])}\rightarrow 0, \text{ as }t_{n_k}\rightarrow +\infty.
 \end{equation}
 Since
 $$(1-x)^3\leq 1-x^3, \text{ for }0\leq x\leq 1,$$
 we have
 \begin{equation}\label{810_1}
   \vert u-u^\star \vert^3\leq \vert u^3-(u^\star)^3 \vert,
 \end{equation}
 which, together with \eqref{810}, implies
  \begin{equation}\label{810_2}
   \|u-u^\star\|_{\Lty([0,1])}\rightarrow 0, \text{ as }t_{n_k}\rightarrow +\infty.
 \end{equation}
 Hence $u$ converges to $u^\star$ in $L^2([0,1])$. Besides, {\blue from the second energy-dissipation inequality \eqref{F01}}, we know $\int_0^1 u^2 \ud h$ is decreasing with respect to $t$ so it has a unique limit $\int_0^1 (u^\star)^2 \ud h$.
Combining this with the uniqueness of critical point in $\dot{H}^2$, we know the stationary constant solution is unique and $f^\star\equiv w^\star\equiv(u^\star)^3$. Therefore, as $t_{n_k}\rightarrow +\infty$, the solution $u^3(t_{n_k})$ converges to the unique constant $(u^\star)^3$ in $H^1([0,1])$.
From the arbitrariness of $t_n$, we know, as $t\rightarrow +\infty$, the solution $u^3$ to PDE \eqref{PDE} converges to $(u^\star)^3$ in $H^1([0,1])$.
Besides, by \eqref{810_2} we obtain \eqref{810_3}.
\end{proof}




\begin{rem}
Given the initial data $u_0,$
we can not obtain a unique value of the constant solution for all weak solutions to PDE \eqref{PDE} so far.
From PDE \eqref{PDE}, the conservation law for classical solution is obvious
\begin{equation}
  \frac{\ud}{\ud t}\int_0^1 \frac{1}{u} \ud h =0, \text{ for any }t\geq 0.
\end{equation}
Hence for any $u_0$, we can calculate the value of the stationary constant solution $u^\star$.
In fact, for $m_0=\int_0^1 \frac{1}{u_0}\ud h$,
we have
 $$(u^\star)^3=\frac{1}{\Big(\int_0^1 \frac{1}{u_0}\ud h\Big)^3}=\frac{1}{m_0^3}.$$
However, the conservation law for weak solution is still an open question although in physics it is true: $u$ is the slope as a function of height and time satisfying
$$\int_0^1 \frac{1}{u}\ud h =\int_0^1 x_h \ud h=x\vert_{h=1}-x\vert_{h=0}\equiv L.$$

\end{rem}


\appendix
\section{Formulations using other physical variables}\label{secapp1}

For completeness, in this appendix, we include some alternative forms of PDE \eqref{PDE} using other physical variables to describe the surface dynamics. To avoid confusion brought by different variables, we replace $h$ by $\alpha$ when the height variable is considered as an independent variable. Let us introduce the following variables:
\begin{itemize}
  \item $u(\alpha,t)$, step slope when considered as a function of surface height $\alpha$;
  \item $\rho(x,t)$, step slope when considered as a function of step location $x$;
  \item $h(x,t)$, surface height profile when considered as a function of step location $x$;
  \item $\phi(\alpha,t)$, step location when considered as a function of surface height $\alpha$.
\end{itemize}

Several straightforward relations between the four profiles are listed as follows.
First, since $\phi$ is the inverse function of $h$ such that
\begin{equation}\label{fan0}
  \alpha=h(\phi(\alpha,t),t), \quad\forall \alpha,
\end{equation}
we have
\begin{equation}\label{fan}
  \phi_t=-\frac{h_t}{h_x},\quad \phi_\alpha=\frac{1}{h_x}.
\end{equation}
Second, from the definitions above, we know
\begin{equation}\label{tri-re}
  u(\alpha,t)=\rho(\phi(\alpha,t),t)=h_x(\phi(\alpha,t),t)=\frac{1}{\phi_\alpha}.
\end{equation}

We formally derive the equations for $h,\,\rho,\,\phi$ from the $u$-equation, which consist with the widely-used $h,\rho$-equation in the previous literature. The four forms of PDEs are rigorously equivalent for local strong solution. Now under the assumption $u\geq 0$, we want to formally derive the other three equations from the $u$-equation \eqref{PDE} (i.e. $u_t=-u^2(u^3)_{\alpha\alpha\alpha\alpha}$ if using variable $\alpha$).

First, from \eqref{tri-re}, we can rewrite \eqref{PDE} as
\begin{equation}\label{07t1}
  \phi_{\alpha t}=\Big(\frac{1}{\phi_\alpha^3}\Big)_{\alpha\alpha\alpha\alpha}.
\end{equation}
Integrating respect to $\alpha$, \eqref{07t1} becomes
\begin{equation}\label{eq:phi0}
    \phi_{ t}=\Big(\frac{1}{\phi_\alpha^3}\Big)_{\alpha\alpha\alpha}+c(t),
\end{equation}
where $c(t)$ is a function independent of $\alpha$ and will be determined later.

Second, let us derive $h$-equation and $\rho$-equation. On one hand, from \eqref{fan} and \eqref{tri-re}, we have
\begin{equation}\label{ut1}
  u_t=\rho_x \phi_t+\rho_t=-\rho\frac{h_t}{h_x}+\rho_t=-\frac{\rho_x}{\rho}h_t+\rho_t.
\end{equation}
On the other hand, due to the chain rule
$u_\alpha=\rho_x \phi_\alpha$,
we have
\begin{equation}
  (u^3)_\alpha=3u^2u_\alpha=3\rho\rho_x=\frac{3}{2}(\rho^2)_x.
\end{equation}
Hence
\begin{equation}\label{ut2}
  \begin{aligned}
    u_t&=-u^2(u^3)_{\alpha\alpha\alpha\alpha}\\
    &=-u^2\Big[\big(((u^3)_{\alpha x}\phi_\alpha)_x \phi_\alpha \big)_x \phi_\alpha\Big]\\
    &=-\frac{3}{2}\rho\Big(\frac{1}{\rho}\big(\frac{(\rho^2)_{xx}}{\rho}\big)_x\Big)_x\\
    &=\frac{3}{2}\frac{\rho_x}{\rho}\Big(\frac{(\rho^2)_{xx}}{\rho}\Big)_x-\frac{3}{2}\Big(\frac{(\rho^2)_{xx}}{\rho}\Big)_{xx}.
  \end{aligned}
\end{equation}
Now denote $A:=-\frac{3}{2}\Big(\frac{(\rho^2)_{xx}}{\rho}\Big)_x$. Comparing \eqref{ut1} with \eqref{ut2}, we have
\begin{equation}
(h_t-A)\frac{\rho_x}{\rho}=(h_t-A)_x,\nn
\end{equation}
which implies
\begin{equation}
  h_t-A=\lambda(t) h_x,\quad \rho_t-A_x=\lambda(t)\rho_x,\nn
\end{equation}
where $\lambda(t)$ is a function independent of $x$ and will be determined later.

Therefore, we know $h$ satisfies
\begin{equation}\label{eq:h0}
  h_t=-\frac{3}{2}\Big(\frac{(h_x^2)_{xx}}{h_x}\Big)_x+\lambda(t) h_x,
\end{equation}
and $\rho$ satisfies
\begin{equation}\label{eq:rho0}
  \rho_t=-\frac{3}{2}\Big(\frac{(\rho^2)_{xx}}{\rho}\Big)_x+\lambda(t) \rho_x.
\end{equation}
From \eqref{eq:h0}, we immediately know $\frac{\ud}{\ud t}\int_0^L h(x) \ud x=0$. Hence we have
\begin{equation}
  \int_0^1 \phi \ud \alpha =L-\int_0^L h(x) dx,\nn
\end{equation}
due to \eqref{fan}.
Thus we know $ \frac{\ud}{\ud t}\int_0^1 \phi \ud \alpha=0.$
This, together with \eqref{eq:phi0}, gives $c(t)=0,$
and we obtain $\phi$-equation
\begin{equation}\label{eq:phin}
    \phi_{ t}=\Big(\frac{1}{\phi_\alpha^3}\Big)_{\alpha\alpha\alpha}.
\end{equation}
Now keep in mind the chain rule $\partial_\alpha= \frac{1}{h_x}\partial_x$ and \eqref{fan}. Changing variable in \eqref{eq:phin} shows that
\begin{equation}
-\frac{h_t}{h_x}=\Big((h_x^3)_x \frac{1}{h_x}\Big)_{\alpha\alpha}=\Big(\frac{3}{2}(h_x^2)_x\Big)_{\alpha\alpha}=\frac{3}{2}\frac{1}{h_x}\Big(\frac{(h_x^2)_{xx}}{h_x}\Big)_x, \nn
\end{equation}
and $  \lambda(t)=0.$
Hence we obtain
$h$-equation
\begin{equation}\label{eq:hn}
  h_t=-\frac{3}{2}\Big(\frac{(h_x^2)_{xx}}{h_x}\Big)_x,
\end{equation}
and $\rho$-equation
\begin{equation}\label{eq:rhon}
  \rho_t=-\frac{3}{2}\Big(\frac{(\rho^2)_{xx}}{\rho}\Big)_{xx}.
\end{equation}
From \eqref{eq:phin}, \eqref{eq:hn} and \eqref{eq:rhon}, we can
immediately see that $\int_0^1 \phi \ud \alpha,$ $\int_0^L h \ud x$ and
$\int_0^L \rho \ud x$ are all constants of motion. The equation (20)
in \cite[p213]{Kohnbook} is exactly \eqref{eq:hn} for vicinal
(monotone) surfaces, which is consistent with our equations.

Now we state the uniqueness and existence result for local strong solution to \eqref{PDE} with positive initial value. The proof for Theorem \ref{local_u} is standard so we omit it here.
\begin{thm}\label{local_u}
Assume $u^0\in \Hper^m([0,1]),$ $u^0\geq \beta$, for some constant $\beta>0$, $m\in \mathbb{Z},\, m\geq 5$. Then there exists time $T_{m}>0$ depending on $\beta,\,\|u^0\|_{\Hper^m([0,1])},$
such that
\begin{equation}
  u\in L^{\infty}([0,T_m];\Hper^{m}([0,1]))\cap L^2([0,T_m];\Hper^{m+2}([0,1])),\nn
\end{equation}
\begin{equation}
u_t\in L^{\infty}([0,T_m];\Hper^{m-4}([0,1])).\nn
\end{equation}
is the unique strong solution of \eqref{PDE} with initial data $u^0$, and $u$ satisfies
\begin{equation}\label{local_1}
u\geq\frac{\beta}{2},\quad \ale  t\in[0,T_m],\,\alpha\in [0,1].
\end{equation}
\end{thm}
From \eqref{local_1} in Theorem \ref{local_u}, we know
\begin{equation}
u(\alpha,t)=\rho(\phi(\alpha,t),t)=h_x(\phi(\alpha,t),t)=\frac{1}{\phi_\alpha}\geq\frac{\beta}{2}>0,\quad \ale  t\in[0,T_m],\,\alpha\in [0,1].\nn
\end{equation}
Hence the formal derivation is mathematically rigorous and we have the
equivalence for local strong solution to \eqref{PDE}, \eqref{eq:hn},
\eqref{eq:rhon} and \eqref{eq:phin}. However, as far as we know, the
rigorous equivalence for global weak solution to \eqref{PDE},
\eqref{eq:hn}, \eqref{eq:rhon} and \eqref{eq:phin} is still open. It
is probably more difficult than the uniqueness of weak solution.


\end{document}